\documentclass{amsart}
\usepackage{amssymb,amsmath,amsthm,amsfonts,amsopn,url,bbm}
\usepackage{enumitem}
\usepackage{amscd, hyperref}
\theoremstyle{plain}
\newtheorem{thm}{Theorem}[section]
\newtheorem{prop}[thm]{Proposition}
\newtheorem{lemma}[thm]{Lemma}
\newtheorem{cor}[thm]{Corollary}
\theoremstyle{definition}
\newtheorem{defn}[thm]{Definition}
\newtheorem*{defn*}{Definition}
\newtheorem{example}[thm]{Example}
\newtheorem{examples}[thm]{Examples}
\theoremstyle{remark}
\newtheorem*{rmk*}{Remark}
\newtheorem{disc}[thm]{Discussion}
\newcommand{\field}[1]{\mathbbm{#1}}

\newcommand{\Z}{\field{Z}}

\newcommand{\ideal}[1]{\mathfrak{#1}}
\newcommand{\m}{\ideal{m}}
\newcommand{\n}{\ideal{n}}
\newcommand{\p}{\ideal{p}}
\newcommand{\q}{\ideal{q}}

\newcommand{\ffunc}[1]{\mathrm{#1}}
\newcommand{\func}[1]{\mathrm{#1} \,}

\newcommand{\Spec}{\func{Spec}}

\newcommand{\Ext}{\ffunc{Ext}}
\newcommand{\Tor}{\ffunc{Tor}}
\DeclareMathOperator{\Ass}{Ass}

\newcommand{\arrow}[1]{\stackrel{#1}{\rightarrow}}
\newcommand{\onto}{\twoheadrightarrow}
\newcommand{\into}{\hookrightarrow}
 
\newcommand{\ra}{\rightarrow}

\DeclareMathOperator{\ann}{ann}
\DeclareMathOperator{\Hom}{Hom}
\DeclareMathOperator{\Min}{Min}

\DeclareMathOperator{\Coass}{Coass}
\DeclareMathOperator{\Att}{Att}
\DeclareMathOperator{\Max}{Max}
\DeclareMathOperator{\Jac}{Jac}
\renewcommand{\phi}{\varphi}

\begin{document}
\title{Criteria for flatness and injectivity}

\author{Neil Epstein}
\address{Universit\"at Osnabr\"uck\\Institut f\"ur Mathematik\\49069 Osnabr\"uck\\Germany}
\email{nepstein@uni-osnabrueck.de}

\author{Yongwei Yao}
\address{
Department of Math and Statistics\\Georgia State University\\30 Pryor St., Atlanta, GA 30303}
\email{yyao@gsu.edu}
\thanks{The second author was partially supported by the National Science Foundation DMS-0700554}

\begin{abstract}
Let $R$ be a commutative Noetherian ring.  We give criteria for flatness of $R$-modules in terms of associated primes and torsion-freeness of certain tensor products.  This allows us to develop a criterion for regularity if $R$ has characteristic $p$, or more generally if it has a locally contracting endomorphism.  Dualizing, we give criteria for injectivity of $R$-modules in terms of coassociated primes and (h-)divisibility of certain $\Hom$-modules.  Along the way, we develop tools to achieve such a dual result.  These include a careful analysis of the notions of divisibility and h-divisibility (including a localization result), a theorem on coassociated primes across a $\Hom$-module base change, and a local criterion for injectivity.
\end{abstract}


\date{December 26, 2010}
\maketitle

\section{Introduction}
The most important classes of modules over a commutative Noetherian ring $R$, from a homological point of view, are the projective, flat, and injective modules.  It is relatively easy to check whether a module is projective, via the well-known criterion that a module is projective if and only if it is locally free.  However, flatness and injectivity are much harder to determine.

It is well-known that an $R$-module $M$ is flat if and only if $\Tor_1^R(R/P, M)=0$ for all prime ideals $P$.  For special classes of modules, there are some criteria for flatness which are easier to check.  For example, a \emph{finitely generated} module is flat if and only if it is projective.  More generally, there is the following \emph{Local Flatness Criterion}, stated here in slightly simplified form (see \cite[Section 22]{Mats} for a self-contained proof):

\begin{thm}[{\cite[10.2.2]{EGA-3.1}}]\label{thm:localflat}
Suppose that $R \ra S$ is a homomorphism of Noetherian rings, $I$ an ideal of $R$ such that $I S \subseteq \Jac(S)$ (the Jacobson radical of $S$), and $M$ is a finite $S$-module.  Then $M$ is flat over $R$ if and only if $M/IM$ is flat over $R/I$ and $\Tor_1^R(R/I,M) = 0$.
\end{thm}

For injectivity, there are very few extant tests.  For general modules, there is the Baer criterion, which says that an $R$-module $M$ is injective if and only if $\Ext^1_R(R/P,M)=0$ for all prime ideals $P$.  It is well-known that a \emph{torsion-free} module $M$ over an integral domain $R$ is injective if and only if it is divisible (and we prove a slightly more general version of this in Corollary~\ref{cor:tfdiv}).  But as most injective modules are not torsion-free, this criterion has limited usefulness.

In this work, we determine criteria for flatness and injectivity by checking torsion-freeness, associated primes, various sorts of divisibility, and coassociated primes, and then reducing to questions of flatness and injectivity over the total quotient ring.  In particular, for reduced rings, we get criteria for flatness in terms of torsion-freeness and associated primes, and criteria for injectivity in terms of (h-)divisibility.

We feel that our criteria for flatness are interesting even in the case where $M$ is a finitely generated $R$-module.  Indeed, the first hint of Theorem~\ref{thm:flatred} came when the authors \cite{nmeYao-HK} were investigating certain generalizations of Hilbert-Kunz multiplicity and needed a criterion to determine when certain direct summands of the $R$-module $R^{1/p^e}$ were free, where $R$ is a reduced $F$-finite Noetherian ring of prime characteristic $p>0$.  Another application of our results to prime characteristic algebra (or any other situation involving a \emph{locally contracting endomorphism}) may be found in Section~\ref{sec:reg}, culminating in Theorem~\ref{thm:frob}.


Here is a description of the contents of the sections to follow.

In Section~\ref{sec:flat}, we recall a standard result about a condition on associated primes satisfied by flat modules, and in Theorems~\ref{thm:flatgeneral} and \ref{thm:flatred}, we show that in some sense this condition mostly characterizes flatness. Several characterizations for flatness there are given, using notions of associated primes, torsion-freeness, and certain tensor products.  We also give an example to show that the condition on associated primes is not enough to guarantee flatness for non-reduced rings.

As discussed above, Section~\ref{sec:reg} gives an application of Theorem~\ref{thm:flatred}.  Indeed, after globalizing the notion of a locally contracting endomorphism, we exhibit in Theorem~\ref{thm:frob} equivalent conditions for regularity of a reduced ring that has a locally contracting endomorphism (\emph{e.g.} any ring of characteristic $p$).

In Section~\ref{sec:div}, we discuss two non-equivalent notions of divisibility of a module: divisible and h-divisible modules.  We recall and demonstrate some dualities with the notion of torsion-freeness, and we show that both of the divisibility notions localize over a ring without embedded primes.

In Section~\ref{sec:Coass}, we recall the notions of attached and coassociated primes.  In Theorem~\ref{thm:CoassHom}, we demonstrate a strong dual (involving coassociated primes and Hom modules) to the standard result from Section~\ref{sec:flat} on associated primes and tensor product.

Our final section gives criteria for injectivity.  Theorem~\ref{thm:localinj} is a dual to the usual local flatness criterion.  Theorem~\ref{thm:injgeneral} gives criteria for injectivity of a module in terms of h-divisibility of Hom modules.  The final result of this paper, Theorem~\ref{thm:injred}, gives a long list of equivalent conditions for injectivity (and a few conditions for being an injective cogenerator) when a ring is reduced, making use of all the concepts discussed in the paper.

\section{Flatness criteria}\label{sec:flat}

First, recall the following standard result (in simplified form below).  Here we use the usual convention that $\Ass 0 = \emptyset$.

\begin{thm}[{\cite[Theorem 23.2]{Mats}}]\label{thm:matflat}
Let $R$ be a Noetherian ring, and let $L,M$ be $R$-modules, where $M$ is flat.  \begin{enumerate}
\item If $P \in \Spec R$ such that $M/PM \neq 0$, then $\Ass(M/PM) = \{P\}$.
\item $\Ass(L \otimes_R M) = \bigcup_{P \in \Ass L} \Ass(M/PM)$.
\end{enumerate}
\end{thm}

For a ring $R$, we say that a module $M$ is \emph{torsion-free} if every zero-divisor on $M$ is a zero-divisor on $R$ (\emph{i.e.} every $R$-regular element is $M$-regular).  It is equivalent to say that the natural map $M \ra M\otimes_R Q$ is injective, where $Q$ is the total quotient ring of $R$.  For a collection $X$ of subsets of a set $S$, we use the notation $\bigcup X$ to mean the union of all the elements of $X$.  In particular, $X \subseteq 2^S$ but $\bigcup X \subseteq S$.

\begin{thm}\label{thm:flatgeneral}
Let $R$ be a Noetherian ring, and $M$ an $R$-module (not necessarily finitely generated).  Let $Q$ be the total quotient ring of $R$.  The following are equivalent: \begin{enumerate}[label=(\alph*)]
\item $M$ is flat.
\item $M \otimes_R Q$ is flat over $Q$, and $\Ass_R (L \otimes_R M) \subseteq \Ass_R L$ for every $R$-module $L$.
\item $M \otimes_R Q$ is flat over $Q$, and $L \otimes_R M$ is torsion-free for every torsion-free $R$-module $L$.
\item $M \otimes_R Q$ is flat over $Q$, and $P \otimes_R M$ is torsion-free for every $P \in \Spec R$.
\item $M \otimes_R Q$ is flat over $Q$, and $\Tor_1^R(R/P, M)$ is torsion-free for every $P \in \Spec R$.
\end{enumerate}
Suppose that $R \ra S$ is a homomorphism of Noetherian rings, $\m \in \Max(R)$, $\m S \subseteq \Jac(S)$, and $M$ is a finite $S$-module.  Then the above conditions are equivalent to: \begin{enumerate}
\item[(d$'$)] $M \otimes_R Q$ is flat over $Q$, and $\m \otimes_R M$ is $R$-torsion-free.
\item[(e$'$)] $M \otimes_R Q$ is flat over $Q$, and $\Tor_1^R(R/\m,M)$ is $R$-torsion-free.
\end{enumerate}
In any case, the following are equivalent: \begin{enumerate}[label=(\roman*)]
\item $M$ is faithfully flat.
\item $M \otimes_R Q$ is flat over $Q$, and $\Ass_R(L \otimes_R M) = \Ass_R L$ for every $R$-module $L$.
\item $M$ is flat and $\Ass_R(L \otimes_R M) = \Ass_R L$ whenever $L$ is a simple $R$-module.
\end{enumerate}
\end{thm}

\begin{proof}
(a) $\implies$ (b): This follows from the base-extension property of flat modules, along with Theorem~\ref{thm:matflat}.

(b) $\implies$ (c): Suppose $L$ is torsion-free.  Let $x$ be a zero-divisor on $L \otimes_R M$.  Then there exists $\p \in \Ass_R (L \otimes_R M)$ with $x \in \p$.  So $x \in \bigcup \Ass_R (L \otimes_R M) \subseteq \bigcup \Ass_R L \subseteq \bigcup \Ass R$.  That is, $x$ is a zero-divisor on $R$.  Thus, $L \otimes_R M$ is torsion-free.

(c) $\implies$ (d): Prime ideals are obviously torsion-free.

(d) $\implies$ (d$'$) and (e) $\implies$ (e'): Trivial.

(d) $\implies$ (e) (or (d$'$) $\implies$ (e$'$)): Fix $P \in \Spec R$ (or fix $P:=\m$).  Consider the following commutative diagram:\[
\begin{CD}
0 @>>> \Tor_1^R(R/P,M) @>d>> P\otimes M 
\\
@. @VhVV                 @VgVV  
\\
0 @>>> \Tor_1^R(R/P,M) \otimes Q @>d\otimes 1_Q>> P\otimes M \otimes Q 
\end{CD}
\]
The top row is exact because $\Tor_1^R(R,M)=0$.  The fact that $P \otimes M$ is torsion-free implies that $g$ is injective.  Thus $g \circ d = (d \otimes 1_Q) \circ h$ is injective, whence $h$ is injective.

(e) (or (e$'$)) $\implies$ (a): To show $M$ is flat, it suffices to show
$\Tor_1^R(R/P, M) = 0$ for all $P \in \Spec(R)$ (or in the situation of (e$'$), the Local Flatness Criterion implies it is enough to do so when $P=\m$). Starting with the short exact sequence $0 \ra P \ra R \ra R/P \ra 0$, tensoring with $M$ gives us the exact sequence \[
0 \ra \Tor_1^R(R/P,M) \ra P \otimes M \ra R \otimes M
\]
Then tensoring with $Q$ gives us the exact sequence \[
\begin{CD}
0 \\
@VVV  \\
\Tor_1^R(R/P,M) \otimes Q \\
@Vd\otimes 1_QVV \\
P\otimes M \otimes Q @= (M \otimes Q) \otimes_Q (Q \otimes P) \\
@Vd' \otimes 1_QVV  @VeVV\\
M \otimes Q @= (M \otimes Q) \otimes_Q (Q \otimes R)
\end{CD}
\]
Since $M \otimes_R Q$ is flat over $Q$ and $Q$ is flat over $R$, the map $e$ (hence also the map $d' \otimes 1_Q$) is injective.  Thus, $ \Tor_1^R(R/P,M)\otimes Q=0$, but since $\Tor_1^R(R/P,M)$ is $R$-torsion-free, $\Tor_1^R(R/P,M)=0$.

Now we prove the equivalence of the conditions for faithful flatness:

(i) $\implies$ (ii): Since $M$ is faithfully flat, $M/\p M \cong R/\p \otimes_R M \neq 0$ for any $\p \in \Ass L$, so by Theorem~\ref{thm:matflat}, $\Ass(L \otimes_R M) = \Ass L$.

(ii) $\implies$ (iii): This follows because we know that (b) implies (a).

Finally, we show that (iii) $\implies$ (i): Let $N \neq 0$ be any $R$-module. Then $N$
contains a non-zero cyclic $R$-module which maps onto a simple
$R$-module, say $L$. By assumption, $\Ass(L \otimes_R M) = \Ass L \neq
\emptyset$, which implies $L \otimes_R M \neq 0$. As $M$ is flat, $L
\otimes_R M \neq 0$ forces $N \otimes_R M \neq 0$.
\end{proof}

We get a particularly nice statement in the case where $R$ is reduced.

\begin{thm}\label{thm:flatred}
Let $R$ be a reduced Noetherian ring, and let $M$ be an $R$-module.  The following are equivalent: \begin{enumerate}[label=(\alph*)] 
\item $M$ is flat.
\item $\Ass_R(L \otimes_R M) \subseteq \Ass_R L$ for every $R$-module $L$.
\item $L \otimes_R M$ is torsion-free for every torsion-free $R$-module $L$.
\item $P \otimes_R M$ is torsion-free for every $P \in \Spec R$.
\item $\Tor_1^R(R/P,M)$ is torsion-free for every $P \in \Spec R$.
\end{enumerate}
If, in addition, $R \ra S$ is a homomorphism of Noetherian rings, $\m \in \Max(R)$, $\m S \subseteq \Jac(S)$, and $M$ is a finite $S$-module, the above conditions are equivalent to: \begin{enumerate}
\item[(d$'$)] $\m \otimes_R M$ is torsion-free.
\item[(e$'$)] $\Tor_1^R(R/\m,M)$ is torsion-free.
\end{enumerate}
In any case, the following are equivalent: \begin{enumerate}[label=(\roman*)]
\item $M$ is faithfully flat.
\item $\Ass_R(L \otimes_R M) = \Ass_R L$ for every $R$-module $L$.
\item $M$ is flat and $\Ass_R(L \otimes_R M) = \Ass_R L$ whenever $L$ is a simple $R$-module.
\end{enumerate}
\end{thm}

\begin{proof}
Since $R$ is reduced, $Q := Q(R)$ is a product of finitely many fields, and hence all $Q$-modules are flat over $Q$.  Thus all implications follow from Theorem~\ref{thm:flatgeneral}.
\end{proof}

Note that the assumptions on flatness over $Q$ in Theorem~\ref{thm:flatgeneral} cannot be omitted.

\begin{example}
Let $R := \Z/(4)$.  Then since $\Spec R = \{2R\}$ has only one element, every nonzero $R$-module $N$ satisfies $\Ass N = \Spec R$.  However, $M := R/(2)$ is not flat over $R$.  To see this, consider the canonical injection $j: 2R \hookrightarrow R$.  Note that $j \otimes_R 1_M$ is the zero map, even though $2R \otimes_R M \cong 2R \neq 0$, which means that $M$ is not $R$-flat.
\end{example}

\section{Regularity of local rings with a locally contracting endomorphism}\label{sec:reg}
We begin this section with a general fact about associated primes via restriction of scalars through a ring homomorphism.  It has been proved by Yassemi in greater generality as \cite[Corollary 1.7]{Ys-assfchg}, but we provide the following, simpler proof for the convenience of the reader.

\begin{prop}\label{pr:ass}
Let $f: R \ra S$ be a homomorphism of commutative rings, where $S$ is Noetherian.  Let $f^*: \Spec S \ra \Spec R$ be the associated map on spectra.  Let $M$ be an $S$-module.  Then $f^*(\Ass_SM) = \Ass_RM$.
\end{prop}

\begin{proof}
First let $\q \in \Ass_SM$.  That is, $\q = \ann_Sz$ for some $z\in M$. Then $f^*(\q) = \ann_Rz$, so that $f^*(\q) \in \Ass_RM$.  Thus, $f^*(\Ass_SM) \subseteq \Ass_RM$.

For the other inclusion, let $\p\in \Ass_RM$.  Then $\p = \ann_Rz$ for some $z\in M$.  Letting $W := R\setminus \p$ and localizing $R$, $S$ and $M$ at $W$, we may assume without loss of generality that $R$ is quasi-local with maximal ideal $\p$.  Now let $N := Sz \subseteq M$.  Clearly $N \neq 0$, so since $S$ is Noetherian, there is some $\q \in \Ass_SN \subseteq \Ass_SM$.  Then $\p \subseteq f^*(\q)$, but since $\p$ is maximal it follows that $\p = f^*(\q)$.  Thus, $f^*(\Ass_SM) \supseteq \Ass_RM$, which finishes the proof.
\end{proof}

\begin{disc}\label{dis}
Take any ring homomorphism $g: R \ra S$ and any $\q \in \Spec S$.  Composing further with the localization map $S \ra S_\q$, we get a ring homomorphism $R \ra S_\q$.  For $\p \in \Spec R$, this latter homomorphism extends to a map $g': R_\p \ra S_\q$, defined by $g'(a/x) = g(a)/g(x)$, \emph{if and only if} $g^{-1}(\q) \subseteq \p$.  Moreover, $g'$ is a \emph{local} homomorphism \emph{if and only if} $\p = g^{-1}(\q)$.  That is, each prime ideal $\q \in \Spec S$ induces a unique local homomorphism $g_\q: R_{g^{-1}(\q)} \ra S_\q$ given by $g_\q(a/x) = g(a)/g(x)$.  Hence in the case where $R=S$, we get an induced local endomorphism $g_\q: R_\q \ra R_\q$ if and only if $\phi^{-1}(\q)=\q$.
\end{disc}

We fix a ring $R$ and a ring endomorphism $\phi: R \ra R$.  For any $R$-module $M$, let ${}^eM$ denote the ($R$-$R$)-\emph{bimodule} which is isomorphic to $M$ as an abelian group, where the image of any $z\in M$ in ${}^eM$ is denoted by ${}^ez$, and whose bimodule structure is given by $a \cdot {}^ez \cdot b := {}^e(\phi^e(a)bz)$ for any $a,b \in R$ and any $z\in M$.  In other words, the \emph{right} module structure is the same as its original structure, but the \emph{left} module structure is via \emph{restriction of scalars} through the endomorphism $\phi^e$.  As a final piece of necessary notation, for any $R$-module $L$, let $F^e(L) := L \otimes_R {}^eR$, considered as a \emph{right} $R$-module.  That is, $F^e(L)$ is the (left) $R$-module whose module structure is that of the right $R$-module $L \otimes_R {}^eR$.  We call $F$ the functor \emph{induced by} $\phi$.

This notation agrees with many authors' notation in the case where $\phi$ is the Frobenius endomorphism.  Indeed, many formulas familiar to characteristic $p$ algebraists still hold, including $F^e(R/I) \cong R / \phi^e(I)R$.

Recall (\emph{e.g.} \cite{AIM-local}) that for a local ring $(R,\m)$, a ring endomorphism $\phi: R \ra R$ is \emph{contracting} if there exists a positive integer $n$ such that $\phi^n(\m) \subseteq \m^2$.  Globalizing this notion, we obtain:
\begin{defn}
Let $\phi: R \ra R$ be a ring endomorphism, and let $X \subseteq \Spec R$.  If for all $\p \in X$ we have an induced local homomorphism $\phi_\p: R_\p \ra R_\p$ (\emph{i.e.} $\phi^{-1}(\p)=\p$) which is contracting, then we say $\phi$ is \emph{locally contracting over $X$}.  If $X=\{\p\}$ for some $\p\in \Spec R$, we say $\phi$ is \emph{locally contracting at $\p$}.  If $X = \Spec R$, we say $\phi$ is \emph{locally contracting}.
\end{defn}

\begin{examples}
If $\phi$ is a contracting (\emph{resp.} locally contracting over $\Max R$) endomorphism, so is $\phi^n$ for any positive integer $n$.  We know of two main classes of locally contracting endomorphisms:
\begin{itemize}
\item Let $R$ be a ring of prime characteristic $p>0$.  Then the \emph{Frobenius endomorphism} $\phi: R \ra R$ defined by $a \mapsto a^p$ is a locally contracting endomorphism.  To see that it induces the identity map on spectra, let $\q\in \Spec R$.  Since $\phi(\q) \subseteq \q^{[p]} \subseteq \q$, we have $\q \subseteq \phi^{-1}(\q)$.  Conversely, suppose $a \in \phi^{-1}(\q)$.  Then $a^p \in \q$, but since $\q$ is a radical ideal, $a \in \q$.
\item Let $k$ be a field, let $X$ be an indeterminate over $k$, and let $R$ be either $k[X]_{(X)}$, $k[\![X]\!]$, or $k[X]/(X^t)$ for some integer $t>0$.  Let $n\geq 2$ be an integer and let $\phi: R \ra R$ be the unique $k$-algebra homomorphism that sends $X \mapsto X^n$.  Then $\phi$ is a locally contracting endomorphism.
\end{itemize}
\end{examples}

\begin{lemma}\label{lem:lcspec}
Let $\phi:R \ra R$ be a locally contracting endomorphism.  Then for all $\p \in \Spec R$, we have $\sqrt{\phi(\p)R} = \p$.  In particular, $\phi^e(\m)R$ is $\m$-primary for all maximal ideals $\m$ and all positive integers $e$.
\end{lemma}

\begin{proof}
The second statement follows from the first, since $\phi^e$ is a locally contracting endomorphism.  So let $\p\in \Spec R$.  By Discussion~\ref{dis}, $\phi(\p)R \subseteq \p$.  On the other hand, take any $\q \in \Spec R$ such that $\phi(\p)R \subseteq \q$.  Then $\p \subseteq \phi^{-1}(\q) = \q$ (again by Discussion~\ref{dis}), which means that $\p$ is the unique minimal prime ideal of $\phi(\p)R$.  That is, $\p=\sqrt{\phi(\p)R}$, as was to be shown.
\end{proof}

E. Kunz proved the following remarkable theorem, showing once again how central the notion of regularity is in commutative algebra:

\begin{thm}[{\cite[Theorem 2.1]{Kunz-regp}}]\label{thm:Kunz}
Let $(R,\m)$ be a Noetherian local ring of prime characteristic $p>0$.  Let $\phi:R \ra R$ be the Frobenius endomorphism, and otherwise use notation as above.  If $R$ is regular, then ${}^eR$ is flat as a left $R$-module for all integers $e>0$.  Conversely, if there is some integer $e>0$ such that ${}^eR$ is flat as a left $R$-module, then $R$ is regular.\footnote{Kunz's original theorem was stated in terms of the flatness of $R$ over its subrings $R^{p^e}$, so he made the additional necessary assumption that $R$ was reduced.  However, in our context, the reducedness of $R$ follows from the assumption that ${}^eR$ is $R$-flat, and hence faithfully flat.  For let $x\in R$ such that $x^{p^e}=0$.  Then $(x) \otimes_R {}^eR = 0$ by flatness of the functor $(-) \otimes_R {}^eR$ applied to the exact sequence $0 \ra (x) \ra R \ra R/(x) \ra 0$, so $x=0$ by faithful flatness.  Thus $R$ is reduced.}
\end{thm}

Avramov, Iyengar, and Miller recently proved a broad generalization, of which we use the following special case:
\begin{thm}[{\cite[part of Theorem 13.3]{AIM-local}}]\label{thm:AIM}
Let $(R,\m)$ be a Noetherian local ring admitting a contracting endomorphism $\phi: R \ra R$.  Then ${}^eR$ is flat as a left $R$-module for some (\emph{resp.} all) $e>0$ if and only if $\phi(\m)R$ is $\m$-primary and $R$ is regular.
\end{thm}

As final preparation for the main theorem of this section, we need the following lemma:
\begin{lemma}\label{lem:tfred}
Let $R$ be Noetherian and $M$ an $R$-module.  Then $M$ is torsion-free if and only if $\bigcup \Ass M \subseteq \bigcup \Ass R$.  If, in particular, $R$ has no embedded primes (\emph{e.g.} if it is reduced), then $M$ is torsion-free if and only if $\Ass M \subseteq \Ass R$, and for any torsion-free $R$-module $M$, $W^{-1}M$ is a torsion-free $(W^{-1}R)$-module for all multiplicative subsets $W$ of $R$.
\end{lemma}

\begin{proof}
The first statement is obvious from the definition.

So suppose $R$ has no embedded primes.  If $\Ass M \subseteq \Ass R$, then clearly $\bigcup \Ass M \subseteq \bigcup \Ass R$.  So suppose $\bigcup \Ass M \subseteq \bigcup \Ass R$.  Let $\p \in \Ass M$.  Then $\p \subseteq \bigcup \Ass R$, so by prime avoidance, there is some prime $\q$ with $\p \subseteq \q \in \Ass R$.  But $\q$ is a minimal prime, so $\p = \q$.

For the final claim, we use the usual bijective correspondence between $\{\q \in \Ass_RN \mid \q \cap W = \emptyset \}$ and $\Ass_{W^{-1}R}(W^{-1}N)$ (given by $\q \mapsto W^{-1}\q$) for $R$-modules $N$.
\end{proof}

\begin{example}
The lack of embedded primes is a necessary condition in Lemma~\ref{lem:tfred}.  To see this, let $R := k[\![x,y,z]\!]/(x^2, xy, xz)$, $N:=R/(x,y)= k[\![x,y,z]\!]/(x,y) \cong k[\![z]\!]$, and $\p := (x,y)R$.  Since every non-unit of $R$ is a zerodivisor, every $R$-module is automatically torsion-free, so that in particular $N$ is torsion-free.  However, the image of $y$ is a zerodivisor on $N_\p$ but not on $R_\p$, so that $N_\p$ is not torsion-free.
\end{example}

As an application of our flatness criterion and the theorems of Kunz and Avramov-Iyengar-Miller, we get the following regularity criterion:

\begin{thm}\label{thm:frob}
Let $R$ be a reduced Noetherian ring admitting a locally contracting\footnote{This theorem is not the most general theorem one can obtain.  For instance, if one takes an endomorphism $\phi$ which is only locally contracting over $\Max R$, and one replaces all the terms $F^e(L)$ (including the case $L=\m$) with $L \otimes_R {}^eR$ (\emph{i.e.} the \emph{left}-module structure), then one still obtains: (b) $\implies$ (c) $\implies$ (d) $\implies$ (e) $\implies$ (a).  If moreover $\phi(\m)R$ is $\m$-primary for all $\m \in \Max R$, then all the statements are equivalent.  However, we stated the theorem in its present form since, in general, the left-module structure of $L \otimes_R {}^eR$ tends to be more complicated than that of $F^e(L)$.}
endomorphism $\phi: R \ra R$, and let $F$ be the functor induced by $\phi$ (\emph{e.g.} if $R$ has positive prime characteristic $p$, we can take $\phi$ and $F$ to be the Frobenius endomorphism and the Frobenius functor respectively).  The following are equivalent: \begin{enumerate}[label=(\alph*)]
\item $R$ is regular.
\item ${}^eR$ is flat as a left $R$-module.
\item For every $e>0$ and every $R$-module $L$, $\Ass F^e(L) = \Ass L$.
\item $\exists e>0$ such that for every $R$-module $L$, $\Ass F^e(L) \subseteq \Ass L$.
\item $\exists e>0$ such that for every $P \in \Spec R$, $F^e(P)$ is torsion-free.
\item $\exists e>0$ such that for every $\m \in \Max R$, $F^e(\m)$ is torsion-free.
\end{enumerate}
\end{thm}

\begin{proof}
First, by Proposition~\ref{pr:ass}, we have for any $R$-module $M$ that $\Ass_R {}^eM = (\phi^e)^*(\Ass_RM)$.  But by Discussion~\ref{dis}, $(\phi^e)^*$ is the identity map, so $(\phi^e)^*(\Ass_RM) = \Ass_RM$.  Also, recall from Lemma~\ref{lem:lcspec} that $\phi^e(\m)R$ is $\m$-primary for all maximal ideals $\m$ and all positive integers $e$.

Next, for each maximal ideal $\m$, we consider the commutative diagram:
\[
\begin{CD}
R @>f>> R_\m \\
@V{\phi^e}VV @V{(\phi^e)_\m}VV\\
R @>f>> R_\m
\end{CD}
\]
Suppose that $R$ is regular.  Then for all maximal ideals $\m$, $R_\m$ is regular.  Hence by Theorem~\ref{thm:AIM} and the fact that $\phi(\m)R$ is $\m$-primary, ${}^e(R_\m)$ is flat as a left $R_\m$-module.  But the commutativity of the diagram above shows that ${}^e(R_\m) \cong ({}^eR)_\m$ as left $R$-modules.  Since $({}^eR)_\m$ is $R$-flat for all maximal ideals $\m$, it follows that ${}^eR$ itself is flat over $R$.  That is, (a) implies (b).

The equivalence of statements (b)-(e) follows from Theorem~\ref{thm:flatred}, since as left $R$-modules, ${}^e(F^e(M)) = M \otimes_R {}^eR$ for any $R$-module $M$ (so that in particular, by the first paragraph of this proof, $\Ass(F^e(L)) = \Ass({}^e(F^e(L))) = \Ass(L \otimes_R{}^e R)$), and since ${}^eR$ is flat over $R$ only if it is faithfully flat.  This last statement follows because for any maximal ideal $\m$, $(R/\m) \otimes_R {}^eR \cong {}^e(R/\phi^e(\m)R) \neq 0$.

As for statement (f), clearly (e) implies (f).  Now suppose (f) is true.  By Lemma~\ref{lem:tfred}, $F^e_R(\m)_\m$ is a torsion-free $R_\m$-module for all $\m$.  But tracing through the commutative diagram above, we see that $F^e_{R_\m}(\m_\m) \cong F^e_R(\m)_\m$, which, as we have already seen, is torsion-free.  Then by Theorem~\ref{thm:flatred}, ${}^e (R_\m)$ is flat over $R_\m$, so that by Theorem~\ref{thm:AIM}, $R_\m$ is regular.  Since this holds for all maximal ideals $\m$, $R$ itself is regular.
\end{proof}

\section{Divisibility}\label{sec:div}

The concept of divisibility of a module has not been used as much as that of torsion-freeness.  Indeed, there are a number of non-equivalent notions!  Most of the work on divisibility has been done in the context of integral domains, as in \cite{FucSal-Sdiv}, but one may extend many of the notions to a much more general context, as was done in \cite{AHT-div}.  Given a commutative ring $R$ and a multiplicative subset $W$ of $R$, let $S = W^{-1}R$.\footnote{In this generality, we have to allow the possibility that $0\in W$, in which case $S=W^{-1}R$ is the zero ring.  But this causes no real problems, as the only $S$-module is $0$, which is injective, flat, divisible, and torsion-free over $S$.}  Following the definitions given by Fuchs and Salce \cite{FucSal-Sdiv} in the case of integral domains, we say in general that an $R$-module $M$ is \emph{$W$-torsion-free} if the multiplication map $w: M \ra M$ is injective for all $w\in W$ (or equivalently, the natural map $M \ra S \otimes_R M$ is injective). We say $M$ is  \emph{$W$-divisible} if the map $w: M \ra M$ is \emph{surjective} for all $w\in W$.  We say $M$ is \emph{h$_W$-divisible} if $M$ is an $R$-quotient of a free $S$-module. In the case where $W$ is the set of all non-zerodivisors of $R$, we use the terms \emph{torsion-free}, \emph{divisible}, and \emph{h-divisible} respectively.  

Note that h$_W$-divisibility implies $W$-divisibility, for if $z\in M$, $w\in W$, and $\pi: F \onto M$ is a surjective $R$-linear map with $F$ a free $(W^{-1}R)$-module and $\pi(u)=z$, then $w g((1/w)u) = z$.

We start this section with a couple of characterizations of h$_W$-divisibility.

\begin{lemma}\label{lem:hdiv}
Let $R$ be a ring, $M$ an $R$-module, $W$ a multiplicative subset of $R$, and $S := W^{-1}R$.  The following are equivalent: \begin{enumerate}[label=(\alph*)]
\item $M$ is h$_W$-divisible
\item The natural evaluation map $e: \Hom_R(S,M) \ra M$ which sends $f \mapsto f(1)$ is surjective.
\item There is a surjective $R$-linear map from an $S$-module onto $M$.
\end{enumerate}
\end{lemma}

\begin{proof}
First suppose $M$ is h$_W$-divisible.  Let $F$ be a free $S$-module such that there is surjective $R$-linear map $p: F \twoheadrightarrow M$.  Pick $z\in M$, and $f\in F$ such that $p(f)=z$.  Since $F$ is an $S$-module, there is a unique $S$-linear map $g: S \ra F$ (which is therefore $R$-linear) such that $g(1)=f$.  Then $(p \circ g): S \ra M$ is an $R$-linear map such that $e(p\circ g) = p(g(1)) = p(f)=z$.  Thus, $e$ is surjective, which means that (a) implies (b).  Clearly (b) implies (c).

To show that (c) implies (a), suppose there is a surjective $R$-linear map $\pi: G \onto M$ from some $S$-module $G$.  Let $p: F \onto G$ be an $S$-linear map where $F$ is a free $S$-module.  Then the map $\pi \circ p: F \onto M$ demonstrates that $M$ is h$_W$-divisible.
\end{proof}

Next, we give a characterization of modules which are \emph{both} $W$- (or h$_W$-) divisible \emph{and} $W$-torsion-free. 
\begin{lemma}\label{lem:localization}
Let $R$ be a ring, $W$ a multiplicative set, and $M$ an $R$-module.  Let $S := W^{-1}R$.  The following are equivalent: \begin{enumerate}[label=(\alph*)]
\item The canonical localization map $g: M \ra W^{-1}M$ is an isomorphism.
\item The $R$-module structure on $M$ extends to an $S$-module structure.
\item The canonical evaluation map $h: \Hom_R(S,M) \ra M$ is an isomorphism.
\item The $R$-module $M$ is $W$-torsion-free and h$_W$-divisible.
\item The $R$-module $M$ is $W$-torsion-free and $W$-divisible.
\end{enumerate}
\end{lemma}

\begin{proof}
(a) $\implies$ (b): Use the $S$-module structure of $W^{-1}M$.

(b) $\implies$ (a): This follows from the universality of the localization map.

(b) $\implies$ (c): Since $S \otimes_R S \cong S$, we have: \begin{align*}
\Hom_R(S,M) &\cong \Hom_R(S, \Hom_S(S,M)) \cong \Hom_S(S \otimes_R S, M) \\
&\cong \Hom_S(S,M) \cong M.
\end{align*}

(c) $\implies$ (b): Use the $S$-module structure of $\Hom_R(S,M)$.

[(a) and (c) together] $\implies$ (d): $g$ is injective and $h$ is surjective.

(d) $\implies$ (e): Any h$_W$-divisible module is $W$-divisible.

(e) $\implies$ (a): Since the injectivity of $g$ is one definition of $W$-torsion-freeness, it suffices to show $g$ is surjective.  Let $z\in M$ and $w\in W$.  By $W$-divisibility of $M$, there is some $y\in M$ with $wy=z$.  Clearly $g(y) = z/w$.
\end{proof}

For an integral domain $R$, Matlis \cite{Mat-cotor} defined an $R$-module $M$ to be \emph{h-divisible} if it was a quotient of an \emph{injective} module, so the casual reader may think there is a conflict of terminology. However, we have the following:

\begin{prop}\label{pr:hdivred}
Let $R$ be a reduced ring with only finitely many minimal primes (\emph{e.g.} a reduced Noetherian ring, or any integral domain).  Let $Q$ be its total ring of fractions, and let $M$ be an $R$-module.   Then $M$ is a quotient of an injective $R$-module (\emph{i.e.} h-divisible in the sense of Matlis) if and only if it is an $R$-quotient of a free $Q$-module (\emph{i.e.} h-divisible in our sense).
\end{prop}

\begin{proof}
Let $N$ be any $Q$-module, free or not.  Since $Q$ is a product of finitely many fields, $N$ is injective over $Q$.  From the natural isomorphism of functors $\Hom_R(-,N) \cong \Hom_Q(Q \otimes_R -, N)$, the $R$-flatness of $Q$ combined with the $Q$-injectivity of $N$ implies that $N$ is $R$-injective.  Thus, any $R$-quotient of $N$ is an $R$-quotient of an injective $R$-module.

Conversely, let $E$ be an injective $R$-module.  There is some free $R$-module $T = \oplus_{i\in I} R$ and a surjection $p: T \onto E$.  Let $F:= \oplus_{i\in I} Q$ and $j: T \into F$ the obvious injection.   Since $E$ is $R$-injective, the map $p: T \ra E$ extends along $j$ to a map $q: F \ra E$, and $q$ is surjective because $p$ is.  Thus, if $M$ is a quotient of the injective $R$-module $E$, it is the $R$-quotient of the free $Q$-module $F$ as well.
\end{proof}

As a corollary, we get an extension to reduced rings of a well-known criterion involving injectivity, torsion-freeness, and divisibility.
\begin{cor}\label{cor:tfdiv}
Let $R$ be a reduced ring with only finitely many minimal primes.  Let $M$ be a torsion-free $R$-module.  The following are equivalent:\begin{enumerate}[label=(\alph*)]
\item $M$ is h-divisible.
\item $M$ is divisible.
\item $M$ is injective.
\end{enumerate}
\end{cor}

\begin{proof}
(a) $\implies$ (b): Any h-divisible module is divisible.

(b) $\implies$ (c): By Lemma~\ref{lem:localization}, $M$ is a $Q$-module, where $Q$ is the total quotient ring of $R$.  But then by the first paragraph of the proof of Proposition~\ref{pr:hdivred}, $M$ is $R$-injective.

(c) $\implies$ (a): This follows from Proposition~\ref{pr:hdivred}.
\end{proof}

Although every h-divisible module is divisible, the converse is often not true.  Matlis \cite{Mat-div} showed that for an integral domain $R$, the two notions are equivalent if and only if $\textrm{projdim}_RQ \leq 1$. On the other hand, Kaplansky \cite{Kap-hdq} showed that if $R$ is a polynomial ring in $n\geq 2$ variables over an uncountable field, then $\textrm{projdim}_RQ \geq 2$.

However, for certain classes of modules, the two notions are always equivalent.  For a Noetherian ring $R$, recall that an \emph{injective cogenerator} is an injective $R$-module $E$ with the property that for any nonzero $R$-module $M$, $\Hom_R(M,E) \neq 0$.  (These always exist. A popular choice is to let $E := \bigoplus_{\m \in \Max(R)} E_R(R/\m)$.)

For an $R$-module $M$, let $M^\vee := \Hom_R(M,E)$, and similarly for morphisms to create a (contravariant) duality functor.  

\begin{prop}[{\cite[Proposition 1.4 and main text]{Ri-cosupp}}]\label{pr:duality}
Let $R$ be a Noetherian ring, $E$ an injective cogenerator for $R$, and $(-)^\vee := \Hom_R(-, E)$ the corresponding duality functor. Let $M$, $N$ be $R$-modules.  \begin{enumerate}
\item For any $R$-linear map $g: M \ra N$, $g$ is injective (\emph{resp.} surjective) if and only if ${g}^\vee$ is surjective (\emph{resp.} injective).
\item $M$ is injective (\emph{resp.} flat) if and only if $M^\vee$ is flat (\emph{resp.} injective).
\item $\Tor_i^R(M,N)^\vee \cong \Ext^i_R(M, N^\vee)$ for all $i\geq 0$.
\item If $M$ is finitely generated, then $\Ext^i_R(M, N)^\vee \cong \Tor_i^R(M,N^\vee)$ for all $i\geq 0$.
\end{enumerate}
\end{prop}

This duality functor helps in examining properties of torsion-freeness and divisibility:

\begin{lemma}
Let $R$ be a Noetherian ring.  Let $W$ be a multiplicative subset of $R$.  Let $E$ be an injective cogenerator for $R$, and $(-)^\vee$ the corresponding duality functor.  Let $L$ be an $R$-module.  We have: \begin{enumerate}[label=(\arabic*)]
\item $L$ is $W$-divisible if and only if $L^\vee$ is $W$-torsion-free.
\item The following are equivalent: \begin{enumerate}[label=(\alph*)]
\item $L^\vee$ is $W$-divisible.
\item $L$ is $W$-torsion-free.
\item $L^\vee$ is h$_W$-divisible.
\end{enumerate}
\item In particular, if $R$ is complete local and $L$ is Noetherian or Artinian, then $L$ is $W$-divisible if and only if it is h$_W$-divisible.
\end{enumerate}
\end{lemma}

\begin{proof}
(1): $L$ is $W$-divisible $\iff L \arrow{w} L$ is surjective for all $w\in W \iff L^\vee \arrow{w} L^\vee$ is injective for all $w\in W$ $\iff L^\vee$ is $W$-torsion-free.

(2): (a) $\iff$ (b): $L^\vee$ is $W$-divisible if and only if $\forall w\in W$, $L^\vee \arrow{w} L^\vee$ is surjective, if and only if $\forall w\in W$, $L \arrow{w} L$ is injective, if and only if $L$ is $W$-torsion-free.

(b) $\implies$ (c): Let $S := W^{-1}R$.  If $L$ is $W$-torsion-free then the natural map $j: L \ra S \otimes_R L$ is injective, whence $j^\vee: (S \otimes_R L)^\vee \ra L^\vee$ is surjective.  But since $(S \otimes_R L)^\vee$ is an $S$-module, Lemma~\ref{lem:hdiv} shows that $L^\vee$ is h$_W$-divisible.

(c) $\implies$ (a): Any h$_W$-divisible module is $W$-divisible.

Statement (3) follows from Matlis duality.
\end{proof}

Finally, it is worth giving an analogue of the localization portion of Lemma~\ref{lem:tfred} for divisible and h-divisible modules:

\begin{prop}\label{pr:divred}
Let $R$ be a Noetherian ring with no embedded primes, and $M$ an $R$-module.  Let $W$ be a multiplicative subset of $R$.  If $M$ is divisible (\emph{resp.} h-divisible) over $R$, then $W^{-1}M$ is divisible (\emph{resp.} h-divisible) over $W^{-1}R$.
\end{prop}

\begin{proof}
If $0\in W$, then $W^{-1}R$ is the zero-ring, so all modules over it are h-divisible.  Thus, we may assume that $0\notin W$.

We begin by setting up some notation: Let $X_1 := \{\p \in \Min R \mid \p \cap W \neq \emptyset \}$, and $X_2 := \Min R \setminus X_1 = \{\p \in \Min R \mid \p \cap W = \emptyset \}$.  Let $S := R \setminus (\bigcup \Min R)$ and $T := R \setminus (\bigcup X_2)$.  We will use the convention that $\bigcup \emptyset = 0$ and $\bigcap \emptyset = R$.

Next, we prove two claims:

\noindent \textbf{Claim 1:}\footnote{The idea for this claim came from the proof of \cite[Lemma 3.3(a)]{AHH}.} Take any $c\in T$ (\emph{i.e.}, $c/1$ is a non-zerodivisor of $W^{-1}R$). We can find $w\in W$ and $r\in R$ such that $wr=0$ and $wc +r$ is a non-zerodivisor of $R$.

\begin{proof}[Proof of Claim 1]
By assumption $\Min R = \Ass R$.

Since the primes in $\Min R$ are mutually incomparable, and since every $\p \in X_1$ intersects $W$, prime avoidance allows us to choose $w \in (\bigcap X_1)\cap W \setminus (\bigcup X_2)$ and $r \in (\bigcap X_2) \setminus (\bigcup X_1)$.  As $c\notin \bigcup X_2$, we have $wc \in (\bigcap X_1) \setminus (\bigcup X_2)$.

Note that $wr \in \bigcap \Min R = \sqrt{0}$, so there is some $n$ with $w^nr^n = 0$.  Replace $r$ and $w$ with $r^n$ and $w^n$ respectively and all the above conditions hold.

Now, for any $\p \in X_1$, note that $wc \in \p$ but $r\notin \p$, so that $wc +r \notin \p$.  Similarly, for any $\q \in X_2$, we have $r \in \q$ but $wc \notin \q$, so that $wc +r \notin \q$.  That is, $wc+r \notin (\bigcup X_1) \cup (\bigcup X_2) = \bigcup \Min R = \bigcup \Ass R$, which means that $wc+r$ is a non-zerodivisor of $R$.
\end{proof}

\noindent \textbf{Claim 2:} $W^{-1}Q(R) \cong Q(W^{-1}R)$, where $Q(-)$ is the function that sends each ring to its total quotient ring.

\begin{proof}[Proof of Claim 2]
Note that $S$, $T$, and $W$ are all unital semigroups

Let $SW$ (\emph{resp.} $TW$) be the unital semigroup generated by $S$ and $W$ (\emph{resp.} $T$ and $W$).  Clearly, we have $W^{-1}Q(R) \cong (SW)^{-1}R$ and $Q(W^{-1}R) \cong (TW)^{-1}R$.  Since $SW \subseteq TW$, the latter ring is a localization of the former.  On the other hand, let $c \in T=TW$ (where equality holds because $W \subseteq T$), and pick $w$, $r$ as in Claim 1.  Then $w^2c = w(wc+r) \in WS$, so that $c/1$ is already invertible in $(SW)^{-1}R = W^{-1}Q(R)$.
\end{proof}

Now suppose $M$ is divisible.  Let $c/1$ be a non-zerodivisor of $W^{-1}R$ (so that $c\in T$), let $w$, $r$ be as in Claim 1, and let $x\in M$.  Then there is some $y\in M$ such that $x=(wc+r)y$.  But then
\[
\frac{wc+r}{1} = \frac{wc}{1} + \frac{wr}{w} = \frac{wc}{1}
\]
in $W^{-1}R$, since $w\in W$ and $wr=0$.  So
\[
\frac{x}{1} = \frac{wc+r}{1} \cdot \frac{y}{1} = \frac{wc}{1} \cdot \frac{y}{1} = \frac{c}{1} \cdot \frac{wy}{1}
\]
in $W^{-1}M$.  That is, $x/1$ is divisible by $c/1$.  Thus, $W^{-1}M$ is divisible as a $(W^{-1}R)$-module.

Finally, suppose $M$ is h-divisible.  Then it is a homomorphic image of a free $Q(R)$-module $F$. But then $W^{-1}F$ is a free module over $W^{-1}Q(R) \cong Q(W^{-1}R)$ (by Claim 2), and the localization at $W$ preserves the surjection onto $M$, so $W^{-1}M$ is a $(W^{-1}R)$-homomorphic image of a free $Q(W^{-1}R)$-module.  Thus $W^{-1}M$ is h-divisible over $W^{-1}R$.
\end{proof}

\section{Coassociated primes and Hom modules along a base change}\label{sec:Coass}
Let $R$ be a commutative Noetherian ring.

MacDonald \cite{Mac-second} and Kirby \cite{Kir-cop} independently proved a dual for Artinian modules to the Lasker-Noether primary decomposition theorem for Noetherian modules.  (In the sequel, we use Macdonald's terminology.)  In this theory, a nonzero $R$-module $M$ is called \emph{secondary} if for all $r\in R$, the map $r: M \ra M$ (left multiplication by $r$) is either surjective or nilpotent.  The annihilator of such a module always has a prime radical, and if $\p = \sqrt{\ann M}$ for a secondary $R$-module $M$, $M$ is called \emph{$\p$-secondary}.  One says that an $R$-module $M$ has a \emph{secondary representation} if one may express it as a finite sum of the form: \[
M = \sum_{i=1}^n M_i,
\]
where each $M_i$ is $\p_i$-secondary, for a set $\{\p_1, \dotsc, \p_n\}$ of pairwise-distinct prime ideals.  Macdonald shows that if such a representation exists, then it has uniqueness properties dual to those enjoyed by primary decompositions.  In particular, if one takes an irredundant such representation, then the set of prime ideals involved is unique.  If such a representation exists, then the $\p_i$ are called the \emph{attached primes} of $M$, and we write $\Att M = \{\p_1, \dotsc, \p_n\}$.  Macdonald and Kirby show that any Artinian module has a secondary representation.

However, not every $R$-module has a secondary representation.  There is, however, a notion of \emph{coassociated primes} ($\Coass$) of an arbitrary $R$-module, by Chambless, dual to that of the associated primes:

\begin{defn}\cite[p. 1134]{Cha-cop} 
Let $R$ be a Noetherian ring, $M$ an $R$-module, and $P \in \Spec R$.  We say that $P$ is \emph{coassociated} to $M$ if there is some nonzero factor module $N$ of $M$ which is \emph{hollow} (\emph{i.e.} $N$ cannot be written as the sum of two proper submodules) such that $P=\{a \in R \mid aN\neq N\}$.  The set of all such primes is written $\Coass M$.
\end{defn}

We collect below some established facts about $\Coass$ and $\Att$:

\begin{prop}\label{pr:Coass}
Let $R$ be a commutative Noetherian ring: \begin{enumerate}
\item If $M \onto N$ is a surjection of $R$-modules, then $\Coass N \subseteq \Coass M$. \cite[Lemma 3 on p. 1136]{Cha-cop}
\item If $0 \ra L \ra M \ra N \ra 0$ is a short exact sequence of $R$-modules, then $\Coass M \subseteq \Coass L \cup \Coass N$. \cite[Lemma 2.1(a)]{Zoe-linkom}
\item $\bigcup \Coass M = \{x\in R \mid x M \neq M\}$. \cite[Folgerung 1 on p. 129]{Zoe-linkom}
\item $\bigcap \Coass M \supseteq \sqrt{\ann M}$. \cite[Folgerung 1 on p. 129]{Zoe-linkom}
\item If $M$ has a secondary representation, then $\Att M = \Coass M$. \cite[Theorem 1.14]{Ys-coass}
\item Let $M$ be a nonzero $R$-module.  Then $\Coass M \neq \emptyset$. \cite[Theorem 1.9]{Ys-coass}
\item Let $\p \in \Spec R$, $\q$ a $\p$-primary ideal, and $E$ an injective $R$-module.  Then $(0 :_E \q)$ is either $0$ or $\p$-secondary. \cite[Lemma 2.1]{Sha-second}
\item Let $g: R \ra S$ be a ring homomorphism, with $S$ Noetherian, and let $M$ be an $S$-module.  Consider $M$ to be an $R$-module via restriction of scalars along $g$.  Then $g^*(\Coass_SM) \subseteq \Coass_RM$.  \cite[Theorem 1.7(i)]{DivTou-coass}
\end{enumerate}
\end{prop}

\begin{thm}\label{thm:CoassHom}
Let $g: R \ra S$ be a map of Noetherian rings.  Let $L$ be a finite $R$-module, and let $M$ be an $S$-module which is injective over $R$ (via the $R$-module structure given by restriction of scalars).
\begin{enumerate}
 \item\label{it:Coassp} For any $\p \in \Spec R$, we have $g^*(\Coass_S (0 :_M \p)) = \Coass_R (0:_M 	\p) \subseteq \{\p\}$.
 \item\label{it:CoassHom} $\Coass_S \Hom_R(L,M) = \bigcup_{\p \in \Ass_RL} \Coass_S (0 :_M \p)$.
\end{enumerate}
\end{thm}

\begin{proof}
To prove (\ref{it:Coassp}) we use parts (5)--(8) of Proposition~\ref{pr:Coass}.  Namely, let $L := (0:_M \p) $.  If $L=0$, then both $\Coass_RL$ and $\Coass_SL$ are empty.  Otherwise, by \ref{pr:Coass}(6), we have $\Coass_SL \neq \emptyset$ and $\Coass_RL \neq \emptyset$.  But then by \ref{pr:Coass}(7), $L$ is $\p$-secondary as an $R$-module, so that by \ref{pr:Coass}(5), $\Coass_RL = \{\p\}$.  Then by \ref{pr:Coass}(8), $\emptyset \neq g^*(\Coass_SL) \subseteq \Coass_RL= \{\p\}$, so that $g^*(\Coass_SL) = \{\p\}$, completing the proof.

It remains to prove (\ref{it:CoassHom}).  For one inclusion, let $\p \in \Ass_RL$. Then there is an injection $0 \ra R/\p \ra L$, which when we apply $\Hom_R(-,M)$, yields an exact sequence \[
\Hom_R(L,M) \ra (0 :_M \p) \ra 0.
\]
Now it follows from Proposition~\ref{pr:Coass}(1) that $\Coass_S (0 :_M \p) \subseteq \Coass_S \Hom_R(L,M)$.

For the reverse inclusion, let $P \in \Coass_S \Hom_R(L,M)$.  Since $L$ is finitely generated and $R$ is Noetherian, setting $\Ass_R L = \{\p_1,\dotsc, \p_n\}$ the $R$-submodule $0\subseteq L$ has a minimal primary decomposition
\[
0 = K_1 \cap \cdots \cap K_n,
\]
such that $L/K_i$ is $\p_i$-coprimary.  We get an injection
\[
0 \ra L \hookrightarrow \bigoplus_{i=1}^n \frac{L}{K_i},
\]
which leads to a surjection
\[
\bigoplus_{i=1}^n \Hom_R(L/K_i,M) \twoheadrightarrow \Hom_R(L,M) \ra 0.
\]
Thus, by Proposition~\ref{pr:Coass}(1 and 2)
\begin{align*}
\Coass_S \Hom_R(L,M) &\subseteq \Coass_S (\bigoplus_{i=1}^n \Hom_R(L/K_i,M)) \\
&= \bigcup_{i=1}^n \Coass_S \Hom_R(L/K_i,M) 
\end{align*}
Since $P \in \Coass_S \Hom_R(L,M)$, we have $P \in \Coass_S \Hom_R(L/K_i,M)$ for some $i$.  Let $L' := L/K_i$ for this choice of $i$, and let $\p := \p_i$.

\noindent \textbf{Claim 1:} $g^{-1}(P) = \p$. (That is, $g^*(P) = \p$.)

\begin{proof}[Proof of Claim 1]
Since $L'$ is $\p$-coprimary, there is some positive integer $k$ such that $\p^k L' = 0$.  Hence $g(\p)^k \Hom_R(L',M) = 0$.  That is, $g(\p) \subseteq \sqrt{\ann_S \Hom_R(L',M)}$.  But then by Proposition~\ref{pr:Coass}(4),
 $g(\p) \subseteq Q$ for every $Q \in \Coass_S \Hom_R(L',M)$, so that in particular, $\p \subseteq g^{-1}(P)$.

Conversely, let $x \in R \setminus \p$.  Since $L'$ is $\p$-coprimary, we get an exact sequence pf $R$-modules
\[
 0 \ra L' \arrow{x} L',
\]
which induces an exact sequence of $S$-modules:
\[
 \Hom_R(L',M) \arrow{g(x)} \Hom_R(L',M) \ra 0
\]
But then by Proposition~\ref{pr:Coass}(3), $g(x)$ is not in any coassociated prime of $\Hom_R(L',M)$.  In particular, $g(x) \notin P$, whence $x \notin g^{-1}(P)$.  This shows that $g^{-1}(P) \subseteq \p$, which completes the proof of the Claim.
\end{proof}

Since $L'$ is finitely generated over the Noetherian ring $R$, there is a prime filtration
\[
 0 = N_0 \subset N_1 \subset \cdots \subset N_t = L',
\]
such that for each $1\leq j \leq t$, $N_j / N_{j-1} \cong R/\q_j$, where the $\q_j$ are prime ideals of $R$.

\noindent \textbf{Claim 2:} For each $1\leq j \leq t$,
\[
\Coass_S \Hom_R(N_j,M) \subseteq \bigcup_{l=1}^j \Coass_S \Hom_R(R/\q_l,M)
\]

\begin{proof}[Proof of Claim 2]
We proceed by induction on $j$.  When $j=1$, we have $N_1 \cong R/q_1$, so there is nothing to prove.

So let $j>1$, and assume the claim for $j-1$.  We have a short exact sequence of $R$-modules
\[
0 \ra N_{j-1} \ra N_j \ra R/\q_j \ra 0,
\]
which induces a short exact sequence of $S$-modules
\[
0 \ra \Hom_R(R/\q_j, M) \ra \Hom_R(N_j,M) \ra \Hom_R(N_{j-1},M) \ra 0.
\]
Thus, by Proposition~\ref{pr:Coass}(2),
\begin{align*}
\Coass_S \Hom_R(N_j,M) &\subseteq (\Coass_S \Hom_R(N_{j-1},M)) \cup \Coass_S \Hom_R(R/\q_j,M) \\
&\subseteq \bigcup_{l=1}^j \Coass_S \Hom_R(R/\q_l,M),
\end{align*}
as was to be shown.
\end{proof}

In particular, since $L'=N_t$,
\[
\Coass_S \Hom_R(L',M) \subseteq \bigcup_{j=1}^t  \Coass_S \Hom_R(R/\q_j,M),
\]
so that since $P \in \Coass_S \Hom_R(L',M)$, there is some $j$ such that
\[
 P \in \Coass_S \Hom_R(R/\q_j,M) = \Coass_S (0:_M \q_j).
\]
But then by part~(\ref{it:Coassp}), $g^{-1}(P) \in \Coass_R(0:_M \q_j) \subseteq \{\q_j\}$, so that $\p = g^{-1}(P) = \q_j$, and $P \in \Coass_S(0:_M \p)$, completing the proof since $\p \in \Ass_RL$.
\end{proof}

\section{Injectivity criteria}\label{sec:inj}

We begin this section with a partial dual of Theorem~\ref{thm:localflat}.  This result may be known to experts; however, we could not find it in the literature, and so we give a proof of it here for completeness.

\begin{thm}[A local injectivity criterion]\label{thm:localinj}
Let $(R,\m,k) \ra (S, \n,\ell)$ be a local homomorphism of Noetherian local rings, and let $M$ be an \emph{Artinian} $S$-module.  Then $M$ is injective over $R$ if and only if $\Ext^1_R(k,M) = 0$.
\end{thm}

\begin{proof}
We only need to prove the ``if'' direction.  So suppose $\Ext^1_R(k, M) = 0$, and assume for the moment that $S=\hat{S}$ -- that is, $S$ is $\n$-adically complete.  Let $E' := E_S(\ell)$.

First note that for any ideal $I$ of $R$, we have \[
\Hom_S(\Ext^1_R(R/I,M), E') \cong \Tor_1^R(R/I, \Hom_S(M,E')).
\]
To see this, consider the (natural) right-exact sequence \[
M \ra \Hom_R(I, M) \ra \Ext^1_R(R/I,M) \ra 0,
\]
then apply the exact functor $\Hom_S(-, E')$, and note that $\Hom_S(\Hom_R(I,M), E') \cong I \otimes_R \Hom_S(M,E')$ in a natural way.

Thus, $\Tor_1^R(k, \Hom_S(M,E')) = \Hom_S(\Ext^1_R(k,M),E') = \Hom_S(0,E')= 0$, but by Matlis duality we have that $\Hom_S(M,E')$ is a finite $S$-module, so by the Local Flatness Criterion, we have that $\Hom_S(M,E')$ is flat as an $R$-module.  In particular,
for any ideal $I$ of $R$, we have \[
0 = \Tor_1^R(R/I, \Hom_S(M, E')) \cong \Hom_S(\Ext^1_R(R/I,M), E').
\]
But since $E'$ is an injective \emph{cogenerator} in the category of $S$-modules, it follows that $\Ext^1_R(R/I,M) = 0$.  Since $I$ was arbitrary, Baer's criterion then shows that $M$ is injective over $R$.

Finally, remove the assumption that $S$ is complete.  Since $M$ is Artinian over $S$, it has a natural structure as an Artinian module over $\hat{S}$, in such a way that the restriction of scalars to $S$ gives the original $S$-module structure of $M$.  Thus, we may replace the map $R \ra S$ with the composite map $R \ra S \ra \hat{S}$, and then by the above argument, it follows that $M$ is injective over $R$.
\end{proof}

\begin{thm}\label{thm:injgeneral}
Let $R$ be a Noetherian ring and $M$ an $R$-module.  Let $Q$ be the total quotient ring of $R$.  The following are equivalent: \begin{enumerate}[label=(\alph*)]
\item $M$ is injective.
\item $\Hom_R(Q,M)$ is injective over $Q$, and $\Hom_R(L,M)$ is h-divisible for every torsion-free $R$-module $L$.
\item $\Hom_R(Q,M)$ is injective over $Q$, and $\Hom_R(P,M)$ is h-divisible for every $P \in \Spec R$.
\end{enumerate}
Moreover, if $(R,\m,k) \ra (S,\n,\ell)$ is a local homomorphism of Noetherian local rings and $M$ is an Artinian $S$-module, then the following condition is also equivalent to the above three:\begin{enumerate}
\item[(c$'$)] $\Hom_R(Q,M)$ is injective over $Q$, and $\Hom_R(\m,M)$ is an h-divisible $R$-module.
\end{enumerate}
\end{thm}

\begin{proof}
(a) $\implies$ (b): For the first statement, recall that if $A \ra B$ is any ring homomorphism and $I$ is an injective $A$-module, then $\Hom_A(B,I)$ is an injective $B$-module.
For the second statement, let $L$ be a torsion-free $R$-module.  Since $L$ is torsion-free, the natural map $L \ra L \otimes_R Q$ is injective.    Since $M$ is an injective module, this implies that the map $C := \Hom_R(L \otimes_R Q, M) \ra \Hom_R(L,M)$ is surjective.  But as $C$ is a $Q$-module, Lemma~\ref{lem:hdiv} now implies that $\Hom_R(L,M)$ is h-divisible.

(b) $\implies$ (c): Every ideal of $R$ is torsion-free.

(c) $\implies$ (c$'$): Obvious.

(c) (or (c$'$)) $\implies$ (a): By a version of Baer's criterion, we need only show that the natural map $\Hom_R(R,M) \ra \Hom_R(P,M)$ is surjective for every $P \in \Spec R$.  (In the situation of (c$'$), it suffices (by Theorem~\ref{thm:localinj}) to show this map is surjective for $P=\m$.  So in that case, we fix the notation $P=\m$ for the remainder of the proof.)  Consider the commutative square:\[ \begin{CD}
\Hom_R(Q, M) @>h>> \Hom_R(Q \otimes_R P, M)\\
@Vg'VV                 @VgVV \\
\Hom_R(R,M) @>f>> \Hom_R(P,M)
\end{CD}
\]
Since the natural map $P\ra R$ is injective and $Q$ is flat over $R$, the corresponding map $Q \otimes_R P \ra Q$ of $Q$-modules is also injective.  Then since $\Hom_R(Q,M)$ is injective over $Q$, we have that \[
\Hom_Q(Q, \Hom_R(Q,M)) \ra \Hom_Q(Q \otimes_R P, \Hom_R(Q,M))
\]
is surjective.  But up to canonical isomorphism, the displayed map may be identified with the map $h$.  Hence $h$ is surjective.

On the other hand, the source module of $g$ is canonically isomorphic to the module $\Hom_R(Q, \Hom_R(P,M))$, in such a way that for any $\alpha: Q \ra \Hom_R(P,M)$, $g(\alpha) = \alpha(1)$.  By Lemma~\ref{lem:hdiv}, the fact that $\Hom_R(P,M)$ is h-divisible means that $g$ is surjective.

We have shown that $g \circ h = f \circ g'$ is surjective, which implies that $f$ is surjective, as required.
\end{proof}

In the reduced case, we get a longer list of equivalent conditions for injectivity, as well as some equivalent conditions for being an injective cogenerator.

\begin{thm}\label{thm:injred}
Let $R$ be a reduced Noetherian ring and let $M$ be an $R$-module.  The following are equivalent: \begin{enumerate}[label=(\alph*)]
\item $M$ is injective
\item $\Coass_R \Hom_R(L,M) \subseteq \Ass_R L$ for every finitely generated $R$-module $L$.
\item $\Hom_R(L,M)$ is h-divisible for every torsion-free module $L$.
\item $\Hom_R(L,M)$ is divisible for every finite torsion-free module $L$.
\item $\Hom_R(P,M)$ is h-divisible for every $P \in \Spec R$.
\item $\Hom_R(P,M)$ is divisible for every $P \in \Spec R$.
\item $\Ext^1_R(R/P,M)$ is divisible for every $P \in \Spec R$.
\end{enumerate}
Moreover, if $(R,\m,k) \ra (S,\n,\ell)$ is a local homomorphism of Noetherian local rings and $M$ is an Artinian $S$-module, then the following conditions are also equivalent to $M$ being injective: \begin{enumerate}
\item[(f$'$)] $\Hom_R(\m,M)$ is a divisible $R$-module.
\item[(g$'$)] $\Ext^1_R(k,M)$ is a divisible $R$-module.
\end{enumerate}
In any case, the following are equivalent: \begin{enumerate}[label=(\roman*)]
\item $M$ is an injective cogenerator.
\item $\Coass_R \Hom_R(L,M) = \Ass_R L$ for every finitely generated $R$-module $L$.
\item $M$ is injective and $\Coass_R \Hom_R(L,M) = \Ass_RL$ whenever $L$ is a simple $R$-module.
\end{enumerate}
\end{thm}

\begin{proof}
Note first that conditions (a), (c) and (e) are equivalent by Theorem~\ref{thm:injgeneral}.

Next, we will show first that (a), (b), (d), (f), (g), (f$'$) and (g$'$) are equivalent:

The fact that (a) $\implies$ (b) follows from Theorem~\ref{thm:CoassHom}.

To see that (b) $\implies$ (d), let $L$ be a torsion-free $R$-module.  Then we have \[
\Coass \Hom_R(L,M) \subseteq \Ass_R L \subseteq \Ass R,
\]
from which it follows that $\bigcup \Coass \Hom_R(L,M) \subseteq \bigcup \Ass R = $ the set of zerodivisors of $R$. But then by Proposition~\ref{pr:Coass}(3), it follows that every non-zerodivisor of $R$ acts surjectively on $\Hom_R(L,M)$.  That is, $\Hom_R(L,M)$ is divisible.

Now let $E := \bigoplus_{\m \in \Max(R)} E_R(R/\m)$, and let $(-)^\vee := \Hom_R(-,E)$ be the associated duality functor.  By Proposition~\ref{pr:duality}(2), $M$ is injective if and only if $M^\vee$ is flat.  For any finite $R$-module $L$, Proposition~\ref{pr:duality}(4) implies that $\Hom_R(L,M)^\vee \cong L \otimes_R M^\vee$ and $\Ext^1_R(L,M)^\vee \cong \Tor_1^R(L, M^\vee)$.  Thus, the equivalence of conditions (a), (d), (f), and (g) follows from Theorem~\ref{thm:flatred} and Proposition~\ref{pr:duality}.  The implication (f$'$) $\implies$ (g$'$) follows from the fact that $\Ext^1_R(k,M)$ is a homomorphic image of $\Hom_R(\m,M)$. Moreover, it is clear that (f) $\implies$ (f$'$) and (g) $\implies$ (g$'$). 

To see that (g$'$) $\implies$ (a) under these conditions, suppose (g$'$) holds.  If $R$ is a field, then $M$ is clearly injective.  Otherwise, let $x\in \m$ be a non-zerodivisor.  Then $x$ acts surjectively on $E := \Ext^1_R(k,M)$ (since $E$ is $R$-divisible), but also $x$ kills $E$ (since $E$ is a $k$-module), so $E=0$.  The conclusion then follows from Theorem~\ref{thm:localinj}.

Finally, we show the equivalence of conditions (i), (ii), and (iii).  To see that (i) $\implies$ (ii), let $\p \in \Ass L$.  Then there is an injection $0 \ra R/\p \ra L$, which when we apply $\Hom_R(-,M)$, yields an exact sequence \[
\Hom_R(L,M) \ra (0 :_M \p) \ra 0.
\]
Now, $\Coass (0 :_M \p) \subseteq \Coass \Hom_R(L,M)$ by Proposition~\ref{pr:Coass}(1), so it suffices to show that $\p \in \Coass (0 :_M \p)$.  We have $(0 :_M\p) = \Hom_R(R/\p, M) \neq 0$ since $M$ is an injective cogenerator.  But then by Proposition~\ref{pr:Coass}(6), $\Coass (0:_M\p) \neq \emptyset$, so that by Theorem~\ref{thm:CoassHom}(1), it follows that $\p \in \Coass (0:_M\p)$.

The implication (ii) $\implies$ (iii) holds because (ii) $\implies$ (b) $\iff$ (a).

Finally, we show that (iii) $\implies$ (i), we need to show that for any nonzero $R$-module $N$, $\Hom_R(N,M) \neq 0$.  The module $N$ contains a cyclic module $C$, which maps onto a simple module $L=R/\m$ for some maximal ideal $\m$.  By assumption, $\Coass \Hom_R(L,M) = \Ass L = \{\m\}$, which certainly implies $\Hom_R(L,M) \neq 0$.  But $\Hom_R(L,M)$ is a submodule of $\Hom_R(C,M)$, and since $M$ injective, $\Hom_R(C,M)$ is a quotient module of $\Hom_R(N,M)$.  Hence, $\Hom_R(N,M) \neq 0$.\end{proof}

\begin{rmk*}
There is, of course, the temptation to adapt the proof of Theorem~\ref{thm:injred} in order to add more equivalent conditions to Theorem~\ref{thm:injgeneral}.  This is indeed possible.  For example, when $R$, $Q$, and $M$ are as in Theorem~\ref{thm:injgeneral}, then $M$ is injective if and only if ``$Q \otimes_R M^\vee$ is flat over $Q$ and $\Hom_R(L,M)$ is divisible for every finitely generated torsion free $R$-module $L$.''  However, such sets of conditions seem too unwieldy to merit theorem status, especially since, as far as the authors know, all such condition sets require explicit use of  the duality functor.  One might hope, for instance, that the flatness of $Q \otimes M^{\vee}$ would be equivalent to the injectivity of $\Hom_R(Q,M)$.  However, one almost never has $\Hom_R(Q, M)^\vee \cong Q \otimes_R M^\vee$, so such equivalences seem unlikely.
\end{rmk*}

\providecommand{\bysame}{\leavevmode\hbox to3em{\hrulefill}\thinspace}
\providecommand{\MR}{\relax\ifhmode\unskip\space\fi MR }
\providecommand{\MRhref}[2]{%
  \href{http://www.ams.org/mathscinet-getitem?mr=#1}{#2}
}
\providecommand{\href}[2]{#2}

\end{document}